\documentclass[12pt]{amsart}
\usepackage{amsmath, amsfonts, amssymb}
\usepackage{graphicx, amsthm, color}
\usepackage{color}

\newtheorem{theorem}{Theorem}[section]
\newtheorem{lemma}[theorem]{Lemma}

\newtheorem{corollary}[theorem]{Corollary}
\newtheorem{proposition}[theorem]{Proposition}
\theoremstyle{definition}

\theoremstyle{remark}
\newtheorem{remark}[theorem]{Remark}
\numberwithin{equation}{section}
\newcommand{\R}{\mathbb R}
\newcommand{\Z}{\mathbb Z}

\def\e{\varepsilon}

\begin{document}

\title[Large time behavior of unbounded solutions]{Large time behavior of unbounded solutions of first-order Hamilton-Jacobi equations
in the whole space}

\author[G. Barles, O. Ley, T.-T. Nguyen, T. V. Phan]{Guy Barles \and Olivier Ley \and Thi-Tuyen Nguyen \and Thanh Viet Phan}

\address{LMPT, F\'ed\'eration Denis Poisson, Universit\'e Fran\c{c}ois-Rabelais Tours, France}
\email{Guy.Barles@lmpt.univ-tours.fr}

\address{IRMAR, INSA Rennes, France}
\email{olivier.ley@insa-rennes.fr}

\address{Dipartimento di Matematica, Universit\`a di Padova, Italy}
\email{ttnguyen@math.unipd.it}

\address{Faculty of Mathematics and Statistics, Ton Duc Thang University, Ho Chi Minh City, Vietnam}
\email{phanthanhviet@tdt.edu.vn}

\begin{abstract}
  We study the large time behavior of solutions of first-order convex Hamilton-Jacobi Equations
  of Eikonal type set in the whole space.
    We assume that the solutions may have arbitrary
    growth.
    A complete study of the structure of solutions of the ergodic problem
    is provided : contrarily to the periodic setting, the ergodic constant
    is not anymore unique, leading to different large time behavior for the solutions.
    We establish the ergodic behavior of the solutions of the Cauchy problem (i) when starting with a bounded from below
    initial condition and (ii) for some particular unbounded from below initial condition, two cases for which we have different ergodic constants which play a role.
    When the solution is not bounded from below, an example showing that the convergence
    may fail in general is provided.
 \end{abstract}


\maketitle

\section{Introduction}

This work is concerned with the large time behavior
for unbounded solutions of the first-order Hamilton-Jacobi equation
\begin{eqnarray} \label{lc-cauchy}
\begin{cases}
 u_t(x,t) + H(x,Du(x,t)) = l(x),
\qquad {\rm in } \ \R^N \times (0, +\infty),\\
u(\cdot,0) = u_0(\cdot) \qquad {\rm in } \ \R^N,
\end{cases}
\end{eqnarray}
where $H\in  W_{\rm loc}^{1,\infty}(\R^N\times \R^N)$ satisfies
\begin{eqnarray}
  && \text{There exists $\nu\in C(\R^N),$ $\nu >0$ such that }
      H(x,p)\geq \nu (x) |p|, \label{Hcoerc}\\
   && 0 =H(x,0) < H(x,p) \ \ \text{ for $p\not= 0$},\label{Hnulle}\\
  && H(x,\cdot) \text{ is convex,} \label{Hconvex}\\
  &&  \text{There exist a constant $C_H>0$ and, for all $R>0$, a constant $k_R$ such that}\nonumber\\
   && |H(x,p)-H(y,q)|\leq k_R(1+|p|)|x-y|+C_H |p-q|,\label{Hlip}\\
   &&  \text{for all $|x|,|y|\leq R$, $p,q\in\R^N.$\nonumber}
\end{eqnarray}

We always assume $u_0, l\in C(\R^N)$ and
\begin{eqnarray}\label{bd-below}
l \geq 0 \quad\text{in $\R^N$}\; .
\end{eqnarray}
These assumptions are those used in the so-called Namah-Roquejoffre case
introduced in~\cite{nr99} in the periodic case, and in Barles-Ro\-que\-jof\-fre~\cite{br06} in the unbounded case.
They are not the most general but, for simplicity, we choose to state as above
since they are well-designed to encompass the classical Eikonal equation
\begin{equation}\label{eik}
u_t(x,t) + a(x)|Du(x,t)| = l(x), \quad \hbox{in  }\R^N\times (0,+\infty),
\end{equation}
where $a(\cdot)$ is a locally Lipschitz, bounded function such that $a(x)>0$ in $\R^N$.
The assumption~\eqref{Hcoerc} is a coercivity assumption, which may be replaced by~\eqref{Hcoerc-gene}.
We also may replace~\eqref{bd-below} by $l$ is bounded from below up to assume that
$H(x,0)- {\rm inf}_{\R^N}l=0$ in~\eqref{Hnulle}. 
\smallskip

Our goal is to prove that, under suitable additional assumptions,
there exists a unique viscosity solution $u$ of~\eqref{lc-cauchy} and that this solution satisfies
\begin{eqnarray*}
&& u(x,t)+ct \to v(x) \ \ \text{in $C(\R^N)$ as $t\to +\infty,$}
\end{eqnarray*}
where $(c,v)\in \R_+\times C(\R^N)$ is a solution to the ergodic problem
\begin{eqnarray}\label{lc-erprob}
&& H(x,Dv(x)) = l(x) + c \qquad \text{in } \R^N.
\end{eqnarray}

This problem has not been widely studied comparing to
the periodic case~\cite{fathi98, nr99, bs00, fs04, ds06, bim13, abil13} and
references therein. The main works in the unbounded setting
are Barles-Roquejoffre~\cite{br06} which extends
the well-known periodic result of Namah-Roquejoffre~\cite{nr99},
the works of Ishii~\cite{ishii08} and Ichihara-Ishii~\cite{ii09}.
A very interesting reference is the review of Ishii~\cite{ishii09}.
We will compare more precisely our results with the existing ones
below but let us mention that our main goal is
to make more precise the large time behavior for the Eikonal Equation~\eqref{eik}
in a setting where the equation is well-posed for solutions with
arbitrary growth, which brings delicate issues. Most of our results
were already obtained or are
close to results of~\cite{br06, ii09} but we use pure PDE
arguments to prove them without using Weak KAM methods and making
a priori assumptions
on the structure of solutions or subsolutions of~\eqref{lc-erprob}.
\smallskip

Changing $u(x,t)$ in $u(x,t)-\inf_{\R^N}\{l\}t$ allows to reduce to the case when $\inf_{\R^N}l=0$ and we are going to actually reduce to that case to simplify the exposure. Taking this into account, we use below the assumption
\begin{eqnarray}\label{lc-argmin-l}
&& \mathop{\rm lim\,inf}_{|x|\to +\infty} l(x) > \inf_{\R^N} (l)=0.
\end{eqnarray}
which is a compactness assumption in the sense that it implies
\begin{eqnarray*}
&& \mathcal{A}:= {\rm argmin}\, l 
  = \{x\in\R^N : l(x)=0\}
\ \text{ is a nonempty compact subset of $\R^N$.} 
\end{eqnarray*}
This subset corresponds to the Aubry set in the framework of Weak KAM theory.

Our first main result collects all the properties we obtain
for the solutions of~\eqref{lc-erprob}.

\begin{theorem}\label{lc-thm:ergodic} (Ergodic problem)\ \\
Assume that $0\leq l\in  C(\R^N)$ and
$H\in C(\R^N\times \R^N).$
\smallskip 

\noindent (i) If $H$ satisfies~\eqref{Hcoerc} and $H(x,0)=0$ for all $x\in \R^N$
then, for all $c\geq 0,$ there exists a solution
$(c,v) \in \R_+ \times W_{\rm loc}^{1,\infty}(\R^N)$ of~\eqref{lc-erprob}.
\smallskip 

\noindent 
(ii) Assume that~\eqref{lc-cauchy} satisfies a comparison principle in $C(\R^N \times [0,+\infty)).$
If $(c,v)$ and $(d,w)$ are solutions of~\eqref{lc-erprob} with $\sup_{\R^N}|v-w| <\infty,$ 
then $c=d$.
\smallskip 

\noindent
(iii) If $\mathcal{A}\not=\emptyset$ and $H$ satisfies~\eqref{Hcoerc} and~\eqref{Hnulle},
then there exists a solution $(c,v) \in \R_+ \times W_{\rm loc}^{1,\infty}(\R^N)$ of~\eqref{lc-erprob}
with $c= 0$ and $v\geq 0.$
If, in addition,
\begin{eqnarray}\label{Hbornesup}
&& \text{$H(x,p)\leq m(|p|)$ for some increasing function $m\in C(\R_+,\R_+)$,}
\end{eqnarray}
and $\mathcal{A}$ satisfies~\eqref{lc-argmin-l},
then $v(x)\to +\infty$ as $|x|\to \infty.$
\smallskip 

\noindent
(iv) Let $c>0.$
If $H$ satisfies~\eqref{Hbornesup}  then any solution
$(c,v)$ of~\eqref{lc-erprob} is unbounded from below.
If $H$ satisfies $H(x,0)=0$ and~\eqref{Hconvex} and $(c,v),$ $(c,w)$
are two solutions of~\eqref{lc-erprob} with $v(x)-w(x)\to 0$ as $|x|\to\infty,$
then $v=w.$
\end{theorem}

\noindent
The situation is completely different with respect to the periodic setting
where there is a unique ergodic constant (or critical value) for which~\eqref{lc-erprob}
has a solution (e.g., Lions-Papanicolaou-Varadhan~\cite{lpv86} or Fathi-Siconolfi~\cite{fs05}). We recover some results of
Barles-Roquejoffre~\cite{br06} and Fathi-Maderna~\cite{fm07}, see Remark~\ref{disc-ergo}
for a discussion. As far as the case of unbounded solutions of elliptic equations is concerned,
let us mention the recent work of Barles-Meireles~\cite{bm16} and the references therein.
\smallskip

Coming back to~\eqref{lc-cauchy}, when $H$ satisfies~\eqref{Hlip},
we have a comparison principle by a ``finite speed of propagation'' type argument, which allows to compare sub- and supersolutions without growth condition (\cite{ishii84, ley01} and Theorem~\ref{lc-compathm}).
It follows that there exists a unique continuous solution defined for all time
as soon as there exist a sub- and supersolution.

\begin{proposition}\label{lc-exis-uniq-cauchy} Assume
that $l\geq 0$ and $H$ satisfies~\eqref{Hcoerc} and~\eqref{Hlip}.
Let $u_0\in C(\R^N)$ and $c\geq 0.$
\smallskip

\noindent (i) There exists a smooth supersolution $(c,v^+)$ of~\eqref{lc-erprob}
satisfying $u_0\leq v^+$ in $\R^N.$~~~~
\smallskip

\noindent (ii) If there exists a subsolution $(c,v^-)$ of~\eqref{lc-erprob}
satisfying $v^-\leq u_0$ in $\R^N,$ then there exists a unique viscosity
solution $u\in C(\R^N\times [0,+\infty))$ of~\eqref{lc-cauchy} such that
\begin{eqnarray}\label{encadrement-u}
&& v^-(x)\leq u(x,t)+ct\leq v^+(x) \qquad\text{for all $(x,t)\in\R^N\times [0,+\infty).$}    
\end{eqnarray}
\end{proposition}

\noindent
Notice that the existence of a subsolution is given by~\eqref{Hnulle} for instance.
\smallskip

We give two convergence results depending on the critical value $c=0$ or $c> 0.$

\begin{theorem}\label{lc-tm:longtime} (Large time behavior starting
with bounded from below initial data)
Assume~\eqref{Hcoerc}-\eqref{Hnulle}-\eqref{Hconvex}-\eqref{Hlip}, $l\geq 0,$ and \eqref{lc-argmin-l}. Then, for every bounded from below initial data
$u_0$,
the unique viscosity solution $u$ of~\eqref{lc-cauchy} satisfies
\begin{eqnarray}\label{cv-thm-main}
u(x,t) \mathop{\to}_{t\to +\infty} v(x)
\quad \text{locally uniformly in $\R^N,$}
\end{eqnarray}
where $(0,v)$ is a solution to~\eqref{lc-erprob}.
\end{theorem}

\begin{theorem}\label{lc-tm:longtime-infty}
  (Large time behavior starting from particular unbounded from below initial data)
  Assume~\eqref{Hcoerc}-\eqref{Hnulle}-\eqref{Hconvex}-\eqref{Hlip}, $l\geq 0$
  and let $(c,v)$ be a solution of~\eqref{lc-erprob} with $c> 0.$
  If there exists a subsolution $(0,\psi)$ of~\eqref{lc-erprob} such that the initial
    data $u_0$ satisfies
\begin{eqnarray}\label{u0-v-to0}
&& \min \{\psi(x),u_0(x)\}- \min \{\psi(x),v(x)\} \mathop{\to} 0 \quad \text{as $|x|\to +\infty$,} 
\end{eqnarray}
then there exists a unique viscosity solution $u$ of~\eqref{lc-cauchy} and 
$u(x,t)+ct \mathop{\to} v(x)$
locally uniformly in $\R^N$ as $t\to +\infty.$
\end{theorem}

\noindent
Let us comment these results. The first convergence result means that,
starting from any bounded from below initial condition (with arbitrary
growth from above), the unique viscosity solution of~\eqref{lc-cauchy}  converges to a
solution $(c,v)$ of the ergodic problem~\eqref{lc-erprob}, which is
given by Theorem~\ref{lc-thm:ergodic}(iii), i.e., 
with
$c=0$ and $v\geq 0,$ $v\to +\infty$ at infinity.
When $u_0$ is not bounded from below, even if it is close to a solution
of the ergodic problem, we give an example
showing that the convergence may fail, see Section~\ref{lc-sec:exple},
where several examples and
interpretations in terms of the
underlying optimal control problem are given.
\smallskip

To describe the second convergence result, suppose that~\eqref{u0-v-to0}
holds with the particular constant subsolution $(0,M)$ for some constant $M$.
In this case,~\eqref{u0-v-to0}
is equivalent to $(u_0-v)(x)\to 0$ when $v(x)\to -\infty.$   
Since, for $c>0$, any solution $(c,v)$ of the ergodic problem
is necessarily unbounded from below (by Theorem~\ref{lc-thm:ergodic}(iv)),
Condition~\eqref{u0-v-to0}
may only happen for unbounded from below initial condition $u_0$.
In this sense, Theorem~\ref{lc-tm:longtime-infty} sheds a new light on
the picture of the asymptotic behavior for~\eqref{lc-cauchy}, bringing a positive result
for some particular unbounded from below initial data.
\smallskip

Theorem~\ref{lc-tm:longtime} and Theorem~\ref{lc-tm:longtime-infty}
generalize and make more precise~\cite[Theorem 4.1]{br06} and~\cite[Theorem 4.2]{br06}
respectively.
In~\cite{br06}, $H$ is bounded uniformly continuous in $\R^N\times B(0,R)$ for any $R>0$ and
$u_0$ is bounded from below and Lipschitz continuous.
Our results are also close to~\cite[Theorem 6.2]{ii09}
as far as Theorem~\ref{lc-tm:longtime} is concerned
and~\cite[Theorem 5.3]{ii09} is very close to Theorem~\ref{lc-tm:longtime-infty}, see Remark~\ref{rmqII}.
In~\cite{ii09}, $H$ may have arbitrary growth with respect to $p$
(\eqref{Hlip} is not required) and the initial condition is bounded from below
with possible arbitrary growth from above. The results apply to more general
equations than ours. The counterpart is that
the unique solvability of~\eqref{lc-cauchy} is not ensured by the assumptions
so {\em the} solution of~\eqref{lc-cauchy} is the one given by the
representation formula in the optimal control framework. The assumptions
are given in terms of existence of particular
sub or supersolutions of~\eqref{lc-erprob}, which may be difficult to check
in some cases.
Finally, let us point out that the proofs of~\cite{br06, ii09} use 
in a crucial way the interpretation of~\eqref{lc-cauchy}-\eqref{lc-erprob}
in terms of control problems and need some arguments of Weak KAM theory.
In this work, we give pure PDE proofs, which are interesting by themselves.
Finally, let us underline that in the arbitrary unbounded setting, we do
not have in hands local Lipschitz bounds, i.e. bounds on $|u_t|, |Du|\leq C,$
with $C$ independent of $t.$ These bounds are easy consequences of
the coercivity of $H$ in the periodic setting and in the Lipschitz setting
of~\cite{br06}. In the general unbounded case, such bounds require
additional restrictive assumptions.
Instead, we provide a more involved proof without further assumptions,
see the proof of Theorem~\ref{lc-tm:longtime}. 
\smallskip

Let us also mention that several other convergence results
are established in~\cite{ishii08} and~\cite{ii09} in the case
of {\em strictly} convex Hamiltonian $H$ and~\cite{ii08}
is devoted to a precise study in the one dimensional case.
We refer again the reader to the review~\cite{ishii09}
for details and many examples.
\smallskip

The paper is organized as follows.
We start by solving the ergodic problem~\eqref{lc-erprob}, see Section~\ref{lc-sec:erg}.
Then, we consider the evolution problem~\eqref{lc-cauchy} in Section~\ref{lc-sec:cauchy}.
Section~\ref{lc-sec:cv} is devoted to the proofs of the theorems
of convergence. Finally, Section~\ref{lc-sec:exple} provides several
examples based both on the Hamilton-Jacobi equations~\eqref{lc-cauchy}-\eqref{lc-erprob}
and on the associated optimal control problem.


\smallskip

\noindent{\bf Acknowledgements:}
Part of this work was made during the stay of T.-T. Nguyen as a Ph.D. student at IRMAR
and she would like to thank  University of Rennes~1 \& INSA for the hospitality.
The work of O. Ley and T.-T. Nguyen was partially supported by the Centre Henri
Lebesgue ANR-11-LABX-0020-01. The authors would like to thank the referees
for the careful reading of the manuscript and their useful comments.

\section{The Ergodic problem}
\label{lc-sec:erg}

Before giving the proof of Theorem~\ref{lc-thm:ergodic}, we start with a lemma
based on the coercivity of $H.$

\begin{lemma}\label{sous-sol-lip}
Let $\Omega\subset\R^N$ be an open bounded subset and $H$ satisfies~\eqref{Hcoerc}.
For every subsolution $(c,v)\in \R_+ \times USC(\overline{\Omega})$
of~\eqref{lc-erprob}, we have $v\in W^{1,\infty}(\Omega)$ and
\begin{eqnarray}\label{apriori-v}
  && |Dv(x)|\leq \max_{y\in\overline{\Omega}}\left\{\frac{l(y)+c}{\nu(y)}\right\}
  \quad \text{for a.e. $x\in \Omega.$}
\end{eqnarray}  
\end{lemma}

\begin{remark}
  Assumption~\eqref{Hcoerc} was stated in that way having in mind the Eikonal Equation~\eqref{eik}
  but it can be replaced by the classical assumption of coercivity
\begin{eqnarray}\label{Hcoerc-gene}
&& \mathop{\rm lim}_{|p|\to +\infty}  \mathop{\rm inf}_{x\in B(0,R)} H(x,p) = +\infty \quad \text{for all $R>0.$}
\end{eqnarray}
\end{remark}

\begin{proof}[Proof of Lemma~\ref{sous-sol-lip}]
Let $B(x_0,R)$ be any ball contained in $\Omega$.
Since $v$ is a viscosity subsolution of
\begin{eqnarray*}
  && |Dv(x)|\leq \max_{y\in\overline{\Omega}}\left\{\frac{l(y)+c}{\nu(y)}\right\}
  \quad \text{in $\Omega$},
\end{eqnarray*}  
we see from \cite[Proposition 1.14, p.140]{abil13}
that $v$ is Lipschitz continuous in $B(x_0,R)$ with the Lipschitz constant
$\max_{\overline{\Omega}}\left\{\frac{l+c}{\nu}\right\}$,
which implies together with Rademacher theorem
\begin{eqnarray*}
  && |Dv(x)|\leq \max_{y\in\overline{\Omega}}\left\{\frac{l(y)+c}{\nu(y)}\right\}
  \quad \text{for a.e. $x\in B(x_0,R)$.}
\end{eqnarray*}  
\end{proof}

We are now able to give the proof of Theorem~\ref{lc-thm:ergodic}.

\begin{proof}[Proof of Theorem~\ref{lc-thm:ergodic}] \
  \smallskip
  
\noindent (i) We follow some arguments of the proof of~\cite[Theorem 2.1]{br06}.
Fix $c\geq 0,$ noticing that $l(x)+c\geq 0$
for every $x\in\R^N$ and recalling that $H(x,0)=0,$ we infer that $0$ is a subsolution of~\eqref{lc-erprob}. 
For $R>0,$ we consider the Dirichlet problem
\begin{eqnarray}\label{pb-dirichl}
&& H(x,Dv)=l(x)+c \ \ \text{in $B(0,R)$,}
\quad v=0 \ \   \text{on $\partial B(0,R).$}
\end{eqnarray}
If $p_R\in \R^N$ and $C_R>0,$ $|p_R|$ are big enough, then, using~\eqref{Hcoerc},
$C_R+\langle p_R,x\rangle$ is a supersolution of~\eqref{pb-dirichl}.
By Perron's method up to the boundary (\cite[Theorem 6.1]{dalio02}),
the function
\begin{eqnarray*}
&&  \hspace*{-2cm}V_R(x):=\mathop{\rm sup}\{v\in USC(\overline{B}(0,R)) \text{ subsolution of~\eqref{pb-dirichl}}:\nonumber\\
&&  \hspace*{2cm}
  0\leq v(x)\leq C_R+\langle p_R,x\rangle \text{ for } x\in \overline{B}(0,R) \},
\end{eqnarray*}
is a discontinuous viscosity solution of~\eqref{pb-dirichl}.
Recall that the boundary conditions are satisfied in the viscosity sense meaning
that either the viscosity inequality or the boundary condition for the
semicontinuous envelopes holds at the boundary. We claim that $V_R\in W^{1,\infty}(\overline{B}(0,R))$
and $V_R(x)=0$ for every $x\in \partial B(0,R),$ i.e., the boundary conditions
are satisfied in the classical sense. At first, from Lemma~\ref{sous-sol-lip},
$V_R\in W^{1,\infty}(B(0,R)).$ By definition, $V_R\geq 0$ in $\overline{B}(0,R),$
so $(V_R)_*\geq 0$ on $\partial B(0,R)$ and the boundary condition holds
in the classical sense for the supersolution. It remains to check that $(V_R)^*\leq 0$ on $\partial B(0,R).$
We argue by contradiction assuming there exists $\hat{x}\in \partial B(0,R)$
such that $(V_R)^*(\hat{x})> 0.$ It follows that the viscosity inequality for subsolutions
holds at $\hat{x},$ i.e., for every $\varphi\in C^1(\overline{B}(0,R))$ such that
$\varphi\geq (V_R)^*$ over $\overline{B}(0,R)$ with $(V_R)^*(\hat{x})=\varphi(\hat{x}),$
we have $H(\hat{x}, D\varphi(\hat{x}))\leq l(\hat{x})+c$ and there exists at least
one such $\varphi.$ Consider, for $K>0,$ $\tilde{\varphi}(x):=\varphi(x)-K\langle\frac{\hat{x}}{|\hat{x}|},x-\hat{x}\rangle.$
We still have $\tilde{\varphi}\geq (V_R)^*$ over $\overline{B}(0,R)$ and $(V_R)^*(\hat{x})=\tilde{\varphi}(\hat{x}).$
Therefore  $H(\hat{x}, D\varphi(\hat{x})-K\frac{\hat{x}}{|\hat{x}|})\leq l(\hat{x})+c$
for every $K>0,$ which is absurd for large $K$ by~\eqref{Hcoerc}.
It ends the  proof of the claim.

We set $v_R(x)=V_R(x)-V_R(0).$ By Lemma~\ref{sous-sol-lip}, for every $R>R',$
we have
\begin{eqnarray}
  && |Dv_R(x)|=|DV_R(x)|\leq C_{R'}:= \max_{\overline{B}(0,R')}\left\{\frac{l+c}{\nu}\right\}
  \quad \text{a.e. $x\in B(0,R'),$}\label{estim-Dv-123}\\
  && |v_R(x)|=|V_R(x)-V_R(0)|\leq C_{R'}R'\quad \text{for $x\in B(0,R').$}\label{estim-v-123}
\end{eqnarray}
Up to an extraction, by Ascoli's Theorem and a diagonal process, $v_R$ converges
in $C(\R^N)$ to a function $v$ as $R\to +\infty,$ which still
satisfies~\eqref{estim-Dv-123}-\eqref{estim-v-123}.
By stability of viscosity solutions, $(c,v)$ is a solution of~\eqref{lc-erprob}. 
\smallskip
  
\noindent (ii)
Let $(c,v)$ and $(d,w)$ be two solutions of~\eqref{lc-erprob} and set
\begin{eqnarray*}
&& V(x,t) = v(x) - ct\\
&& W(x,t) = w(x) -dt. 
\end{eqnarray*}
To show that $c=d$, we argue by contradiction, assuming that $c < d$.
Obviously, $V$ is a viscosity solution of~\eqref{lc-cauchy} with $u_0 = v$ 
and $W$  is  a viscosity solution of~\eqref{lc-cauchy} with $u_0 = w$.
Using the comparison principle for~\eqref{lc-cauchy}, we get that
$$
V(x,t) - W(x,t) \leq \sup_{\R^N}\{v - w\} \qquad \text{ for all } (x,t) \in \R^N \times [0,\infty).
$$
This means that
$$
(d-c)t + v(x) - w(x) \leq \sup_{\R^N}\{v - w\} \qquad \text{ for all } (x,t) \in \R^N \times [0,\infty).
$$
Recalling that $\sup_{\R^N}|v-w| < \infty$,
we get a contradiction for $t$ large enough.
By exchanging the roles of $v,w,$ we conclude that $c=d.$
\smallskip
  
\noindent(iii) Let $c=0.$ We apply the Perron's method using in a crucial way $\mathcal{A}\not=\emptyset.$
Let $\mathcal{S}=\{ w\in USC(\R^N) \text{ subsolution of~\eqref{lc-erprob}}:
0\leq w \text{ and } w=0 \text{ on } \mathcal{A}\}$ and set
$$
v(x):=\mathop{\rm sup}_{w\in \mathcal{S}} w(x).
$$
Noticing that $l+c\geq 0$ and since $H(x,0)=0,$ we have $0\in \mathcal{S}.$ 
Let $x\in\R^N$ and $R>0$ large enough such that $x\in B(0,R)$ and
there exists $x_\mathcal{A}\in B(0,R)\cap \mathcal{A}.$
For all $w\in \mathcal{S},$ by Lemma~\ref{sous-sol-lip}, we have
\begin{eqnarray*}
&& 0\leq w(x)\leq w(x_\mathcal{A})+\max_{\overline{B}(0,R)}\left\{\frac{l+c}{\nu}\right\}\ |x-x_\mathcal{A}|\leq 2R \max_{\overline{B}(0,R)}\left\{\frac{l+c}{\nu}\right\},
\end{eqnarray*}  
since $w(x_\mathcal{A})=0.$ The above upper-bound does not depend on $w\in \mathcal{S},$ so we deduce that
$0\leq v(x)<+\infty$ for every $x\in\R^N.$

We claim that $v$ is a solution of~\eqref{lc-erprob}. At first, by classical arguments (\cite{barles94}), $v$ is still a subsolution
of~\eqref{lc-erprob} satisfying $v\geq 0$ in $\R^N$ and $v=0$ on $\mathcal{A}.$ By Lemma~\ref{sous-sol-lip},
$v\in W_{\rm loc}^{1,\infty}(\R^N).$ To prove that $v$ is a supersolution, we argue as usual by contradiction assuming that
there exists $\hat{x}$ and $\varphi\in C^1(\R^N)$ such that $\varphi\leq v,$ $v(\hat{x})=\varphi(\hat{x})$ and the
viscosity supersolution inequality does not hold, i.e., $H(\hat{x}, D\varphi(\hat{x}))< l(\hat{x})+c.$
To reach a contradiction, one slightly modify $v$ near $\hat{x}$ in order to build a new subsolution $\hat{v}$ in $\mathcal{S},$
which is strictly bigger than $v$ near $\hat{x}.$ To be able to proceed as in the classical proof, it is enough
to check that $\hat{x}\not\in \mathcal{A}$; otherwise $\hat{v}$ will not be 0 on $\mathcal{A}$ leading to $\hat{v} \not\in \mathcal{S}.$
If $\hat{x}\in \mathcal{A},$ then $l(\hat{x})+c=0.$ By~\eqref{Hnulle}, we obtain
$0\leq H(\hat{x}, D\varphi(\hat{x}))< l(\hat{x})+c=0,$ which is not possible.
It ends the proof of the claim.

From~\eqref{lc-argmin-l}, there exists $\epsilon_\mathcal{A}, R_\mathcal{A} >0$
such that $l(x) > \min _{\R^N}l+\epsilon_\mathcal{A}$ for all $x\in \R^N\setminus B(0,R_\mathcal{A}).$
By~\eqref{Hbornesup},
$v$ satisfies, in the viscosity sense
\begin{eqnarray*}
m(|Dv|) \geq H(x,Dv)\geq l(x)+c \geq \epsilon_\mathcal{A} \quad \text{in $\R^N\setminus B(0,R_\mathcal{A}).$}
\end{eqnarray*}
Therefore, for all $x\in\R^N$ and every $p$ in the viscosity subdifferential
$D^- v(x)$ of $v$ at $x,$ we have $|p|\geq m^{-1}(\epsilon_\mathcal{A}) >0.$
By the viscosity decrease principle~\cite[Lemma 4.1]{ley01},
for all $B(x,R)\subset \R^N\setminus B(0,R_\mathcal{A}),$ we obtain
\begin{eqnarray*}
&& \mathop{\rm inf}_{B(x,R)} v \leq v(x)- m^{-1}(\epsilon_\mathcal{A}) R.
\end{eqnarray*}
Since $v\geq 0,$ for any $R>0$ and $x$ such that $|x|> R_\mathcal{A}+R,$
we conclude $v(x)\geq m^{-1}(\epsilon_\mathcal{A}) R,$ which proves that
$v(x)\to +\infty$ as $|x|\to +\infty.$
\smallskip

\noindent(iv) Since $c> 0,$ there exists $\alpha >0$
such that $l(x)+c\geq \alpha$ for all $x\in\R^N.$

To prove that $v$ is unbounded from below, we use again the
viscosity decrease principle~\cite[Lemma 4.1]{ley01}.
By~\eqref{Hbornesup},
$v$ satisfies, in the viscosity sense
\begin{eqnarray*}
m(|Dv|) \geq H(x,Dv)\geq \alpha \quad \text{in $\R^N,$}
\end{eqnarray*}
which implies, for all $R>0,$
\begin{eqnarray*}
  \mathop{\rm inf}_{B(0,R)} v \leq v(0)-  m^{-1}(\alpha) R
\end{eqnarray*}
and so $v$ cannot be bounded from below.

For the second part of the result, we argue by contradiction assuming
that $v\not= w.$ Without loss of generality, there exists $\eta >0$ and
$\hat{x}\in\R^N$ such that $(v-w)(\hat{x})>3\eta.$
Since $(v-w)(x)\to 0$ as $|x|\to +\infty,$ there exists $R>0$
such that $|(v-w)(x)|<\eta$ when $|x|\geq R.$  Up to choose $0<\mu <1$ sufficiently close
to 1, we have $|(\mu v-w)(\hat{x})|> 2\eta$ and,
by compactness of $\partial B(0,R),$
$|(\mu v-w)(x)|< 2\eta$ for all $x\in \partial B(0,R).$
It follows that $M:=\max_{\overline{B}(0,R)}\mu v-w$
cannot be achieved at the boundary of $\overline{B}(0,R).$
Consider
\begin{eqnarray*}
M_\e :=  \max_{x,y\in \overline{B}(0,R)} \left\{\mu v(x)-w(y)-\frac{|x-y|^2}{\e^2}\right\},
\end{eqnarray*}
which is achieved at some $(\bar{x},\bar{y}).$
By classical properties (\cite{bcd97, barles94}), up to extract some subsequences $\epsilon\to 0,$
\begin{eqnarray*}
&& \frac{|\bar{x}-\bar{y}|^2}{\epsilon^2}\to 0, \\
&& \bar{x}, \bar{y} \to x_0 \quad\text{for some $x_0\in \overline{B}(0,R),$}\\
&& M_\e \to M.
\end{eqnarray*}
It follows that $M= (\mu v-w )(x_0)$ and therefore, for $\e$ small enough,
neither $\bar{x}$ nor $\bar{y}$ is on the boundary of  $\overline{B}(0,R).$
We can write the viscosity inequalities for $v$ subsolution at $\bar{x}$
and $w$ supersolution at $\bar{y}$ for small $\e$ leading to
\begin{eqnarray*}
  && H(\bar{x},\frac{\bar{p}}{\mu})\leq l(\bar{x})+c,\\
  && H(\bar{y},\bar{p})\geq l(\bar{y})+c,
\end{eqnarray*}
where we set $\displaystyle \bar{p}=2\frac{(\bar{x}-\bar{y})}{\e^2}.$
Noticing that
\begin{eqnarray*}
 && \mu v(\bar{x})- w(\bar{x})\leq \mu v(\bar{x})- w(\bar{y})-\frac{|\bar{x}-\bar{y}|^2}{\e^2}
\end{eqnarray*}
and using that $w$ is Lipschitz continuous with some constant $C_R$ in $\overline{B}(0,R)$
by Lemma~\ref{sous-sol-lip}, we obtain $|\bar{p}|\leq C_R.$ Therefore, up to extract
a subsequence $\e\to 0,$ we have $\bar{p}\to p_0.$
By the convexity of $H$,
\begin{eqnarray*}
&& H(\bar{x}, p)=H(\bar{x}, \mu \frac{\bar{p}}{\mu}+(1-\mu)0)
  \leq \mu H(\bar{x}, \frac{\bar{p}}{\mu})+(1-\mu) H(\bar{x}, 0).
\end{eqnarray*}
Using $H(\bar{x},0)=0,$ we get
\begin{eqnarray*}\label{convexityH}
&& 0\leq \mu H(\bar{x}, \frac{\bar{p}}{\mu})- H(\bar{x}, \bar{p}).
\end{eqnarray*}
Subtracting the viscosity inequalities and using
the above estimates, we obtain
\begin{eqnarray*}
&& 0\leq  \mu H(\bar{x},\frac{\bar{p}}{\mu})-H(\bar{x},\bar{p})
\leq  H(\bar{y},\bar{p})-H(\bar{x},\bar{p})+\mu (l(\bar{x})+c)-(l(\bar{y})+c).
\end{eqnarray*}
Sending $\e\to 0,$ we reach $0\leq (\mu -1)(l(x_0)+c)\leq (\mu -1) \alpha <0,$
which is a contradiction. It ends the proof of the theorem.
\end{proof}

\begin{remark}\label{disc-ergo} \ \\
(i) In the periodic setting, there is a unique $c= 0$ such that~\eqref{lc-erprob}
has a solution. It is not anymore the case in the unbounded setting where
there exist solutions for all $c\geq 0.$
The proof 
is adapted from~\cite[Theorem 2.1]{br06}. Similar issues are
studied in~\cite{fm07}.
Notice that, when $c< 0,$ there is no
subsolution (thus no solution)
because of~\eqref{Hnulle}.\\
(ii) In the periodic setting, the classical proof of existence of a solution
to~\eqref{lc-erprob} (\cite{lpv86}) uses the auxiliary approximate equation
\begin{eqnarray}\label{eq-lambda}
&& \lambda v^\lambda + H(x, Dv^\lambda )=l(x) \quad\text{in $\R^N.$}   
\end{eqnarray}
In our case, it gives only the existence of a solution $(c,v)$ with
$c=  0$ but not for all $c\geq 0.$
\\
(iii) Neither the proof using~\eqref{eq-lambda}, nor the proof of Theorem~\ref{lc-thm:ergodic}(i)
using the Dirichlet problem~\eqref{pb-dirichl} yields a nonnegative (or bounded
from below) solution $v$ of~\eqref{lc-erprob} for  $c= 0$.
See Section~\ref{sec:exple-sol-ergo}
for an explicit computation of the solution of~\eqref{pb-dirichl}. It is why
we need another proof to construct such a solution. See~\cite{bm16} for the
same result in the viscous case.
\\
(iv) For $c= 0,$ bounded solutions to the ergodic problem
may exist, e.g., when $l$ is periodic (\cite{lpv86} and the example in
Remark~\ref{rmk-perio}). If $\mathcal{A}$
is bounded, we can prove with similar arguments as in the proof of the theorem
that all solutions of the ergodic problem are unbounded.
\\
(v) When $c> 0,$ there is no bounded solution to~\eqref{lc-erprob}
even if $l$ is periodic or bounded.
\\
(vi) Theorem~\ref{lc-thm:ergodic} does not require $H$ to satisfy~\eqref{Hlip}
so it applies to more general equations than~\eqref{eik}, for instance with quadratic
Hamiltonians.
\\
(vii) The assumption that a comparison principle in $C(\R^N\times [0,+\infty))$ holds for~\eqref{lc-cauchy}
in Theorem~\ref{lc-thm:ergodic}(ii)
may seem to be a strong assumption but it is true for the Eikonal equation,
i.e., when $H$ satisfies~\eqref{Hlip}, see Theorem~\ref{lc-compathm}.
In this case, $H$ automatically satisfies~\eqref{Hbornesup} with
$m(r)=C_H r.$
\end{remark}

\section{The Cauchy problem}
\label{lc-sec:cauchy}

In this section we study the Cauchy problem~\eqref{lc-cauchy}.
We start with some comments about Proposition~\ref{lc-exis-uniq-cauchy}
and then we prove it.

Existence and uniqueness are
  based on the comparison Theorem~\ref{lc-compathm} without growth condition,
  which holds when~\eqref{Hlip} is satisfied thanks to the finite speed of propagation.
  When $u_0$ is bounded from below and~\eqref{Hnulle} holds, $\mathop{\rm inf}_{\R^N}u_0$ is a
  subsolution of~\eqref{lc-cauchy} and~\eqref{encadrement-u} takes the simpler form
\begin{eqnarray*}
&& \mathop{\rm inf}_{\R^N}u_0 \leq u(x,t)+ct\leq v^+(x).   
\end{eqnarray*}


\begin{proof}[Proof of Proposition~\ref{lc-exis-uniq-cauchy}] \ \\
(i) Let
\begin{eqnarray*}
&& {v^+(x):= f_0(|x|)+\int_0^{|x|} f_1(s)ds, }
\end{eqnarray*}
where
\begin{eqnarray*}
  &&
\left\{
\begin{array}{l}
  \text{$f_0:\R_+\to\R_+$  $C^1$ nondecreasing, $f_0'(0)=0$ and $f_0(|x|)\geq u_0(x)$}\\
  \text{$f_1:\R_+\to\R_+$ continuous, $f_1(0)=0$ and
    $\displaystyle f_1(|x|)\geq \frac{l(x)+c}{\nu(x)},$}\\
\end{array}
\right.
\end{eqnarray*}
where $\nu$ appears in~\eqref{Hcoerc}.

The existence of such functions $f_0, f_1$ is classical (see~\cite[Proof of Theorem 2.2]{barles94} for instance).
It is straightforward to see that $v^+\in C^1(\R^N),$ $v^+\geq u_0$ and $(c,v^+)$ is a supersolution
of~\eqref{lc-erprob} thanks to~\eqref{Hcoerc}.

\noindent (ii)
It is obvious that $(c,v)$ is a solution (respectively a subsolution, supersolution) of~\eqref{lc-erprob}
if and only if $V(x,t)=v(x)-ct$ is a solution (respectively a subsolution, supersolution)
of~\eqref{lc-cauchy} with initial data $V(x,0)=v(x).$
{We have $v^-\leq u_0\leq v^+,$ where $v^-$ is the subsolution given by assumption
  and $v^+$ is the supersolution built in (i).}
Using Perron's method and Theorem~\ref{lc-compathm}, which holds thanks to~\eqref{Hlip},
we conclude that there exists a unique viscosity
solution $u\in C(\R^N\times [0,+\infty))$ of~\eqref{lc-cauchy} such that
\begin{eqnarray*}
&& v^-(x)-ct\leq u(x,t)\leq v^+(x)-ct.    
\end{eqnarray*}
\end{proof}

\section{Large time behavior of solutions}
\label{lc-sec:cv}

\subsection{Proof of Theorem~\ref{lc-tm:longtime}}

We first consider the case when $u_0$ is bounded. Recalling that $c=0$ and $u$ is solution
of~\eqref{lc-cauchy}, we see by Proposition~\ref{lc-exis-uniq-cauchy} that
\begin{eqnarray}\label{borneLinfini}
&& \mathop{\rm inf}_{\R^N}u_0 \leq u(x,t)\leq v^+(x),   
\end{eqnarray}
where $(0,v^+)$ is a supersolution of~\eqref{lc-erprob} satisfying $v^+\geq u_0.$

The first step is to obtain better estimates for the large time behavior of $u$. To do so, we consider $(0,v_1)$ and $(0,v_2)$ two solutions of~\eqref{lc-erprob}. Such solutions exist from Theorem~\ref{lc-thm:ergodic}(iii) with $c=0$ and $\mathcal{A}\not=\emptyset.$ Moreover, $v_1(x), v_2 (x)\to +\infty$ as $|x|\to +\infty$ since  $\mathcal{A}$ is supposed to be compact and~\eqref{Hbornesup} holds because
of Assumptions~\eqref{Hnulle} and~\eqref{Hlip}.

We have
\begin{lemma}\label{IneqLT} There exist two constants $k_1,k_2\geq 0$ such that
$$ v_1(x) -k_1 \leq \liminf_{t \to +\infty} u(x,t) \leq \limsup_{t \to +\infty} u(x,t)\leq v_2(x)+k_2\quad\hbox{in  }\R^N .$$
As a consequence, for any solutions $(0,v_1)$ and $(0,v_2)$ of~\eqref{lc-erprob}, $v_1-v_2$ is bounded.
\end{lemma}

\begin{proof}[Proof of Lemma~\ref{IneqLT}]  The proof of third inequality in Lemma~\ref{IneqLT} is obvious : since $u_0$ is bounded and $v_2 (x)\to +\infty$ as $|x|\to +\infty$, there exists $k_2$ such that $u_0\leq v_2+k_2 $ in $\R^N.$ Then, by comparison (Theorem~\ref{lc-compathm})
$$ u(x,t) \leq v_2(x)+k_2\quad\hbox{in  }\R^N\times [0,+\infty)\; ,$$
which implies the $\limsup$-inequality.

The $\liminf$-one is less standard. Let $R_\mathcal{A}>0$ be such that $\mathcal{A}\subset B(0,R_\mathcal{A}/2)$ and set
\begin{eqnarray*}
  && C_1=C_1(\mathcal{A},v_1):= \mathop{\rm sup}_{\overline{B}(0,R_\mathcal{A})} v_1 +1.
\end{eqnarray*}
Notice that, by definition of $\mathcal{A}$ and~\eqref{lc-argmin-l},
there exists $\eta_\mathcal{A} >0$ such that
\begin{eqnarray} \label{prop-hors-A}
  && l(x)\geq \eta_\mathcal{A} >0 \quad \text{for all $x\in \R^N\setminus B(0,R_\mathcal{A}).$}
\end{eqnarray}
Using that ${\rm min}\{v_1,C_1\}$ is bounded from above and $u_0$ is
bounded, there exists $k_1=k_1(\mathcal{A},v_1,u_0)$ such that
\begin{eqnarray}\label{ineg720}
  && {\rm min}\{v_1,C_1\}-k_1 \leq u_0 \quad \text{in $\R^N.$}
\end{eqnarray}

Next we have to examine the large time behavior of the solution associated to the initial condition ${\rm min}\{v_1,C_1\}-k_1$ and to do so, we use the following result of Barron and Jensen (see Appendix).
\begin{lemma}\label{barron-jensen} \cite{bj90}
  Assume~\eqref{Hconvex} and let $u,\tilde{u}$ be locally Lipschitz subsolutions (resp. solutions) 
  of~\eqref{lc-cauchy}.
  Then  ${\rm min} \{u,\tilde{u}\}$ is still a subsolution (resp. a solution) of~\eqref{lc-cauchy}.
\end{lemma}
To use it, we remark that the function $w^-(x,t):= C_1+\eta_\mathcal{A} t$ is a smooth subsolution of~\eqref{lc-cauchy}
in $(\R^N\setminus \overline{B}(0,R_\mathcal{A}))\times (0,+\infty).$ Indeed, for all $|x|>R_\mathcal{A},$ $t>0,$
\begin{eqnarray*}
&& w_t^- +H(x,Dw^-)= \eta_\mathcal{A} +H(x,0)=\eta_\mathcal{A} \leq l(x)
\end{eqnarray*}
using~\eqref{Hnulle} and~\eqref{prop-hors-A}. Since $v_1$ is a locally Lipschitz continuous subsolution of~\eqref{lc-cauchy}
in $\R^N\times (0,+\infty)$, we can use Lemma~\ref{barron-jensen} in $(\R^N\setminus \overline{B}(0,R_\mathcal{A}))\times (0,+\infty)$ to conclude that $\min\{v_1,w^-\}-k_1$ is a subsolution, while in a neighborhood of $\overline{B}(0,R_\mathcal{A})\times (0,+\infty)$, we have $\min\{v_1,w^-\}-k_1=v_1-k_1$ by definition of $C_1$.

Then, by comparison (Theorem~\ref{lc-compathm})
$$ \min\{v_1(x),C_1+\eta_\mathcal{A} t\}-k_1 \leq u(x,t) \quad\hbox{in  }\R^N\times [0,+\infty)\; ,$$
and one concludes easily.

The last assertion of Lemma~\ref{IneqLT}
is obvious since $v_1, v_2$ are arbitrary solutions of~\eqref{lc-erprob} and we can exchange their roles.
\end{proof}

The next step of the proof of Theorem~\ref{lc-tm:longtime}
consists in introducing the half-relaxed limits~\cite{cil92, barles94}
$$
\underline{u}(x)=\mathop{{\rm lim\,inf}_*}_{t\to +\infty} u(x,t), \qquad
\overline{u}(x)=\mathop{{\rm lim\,sup}^*}_{t\to +\infty} u(x,t).
$$
They are well-defined for all $x\in\R^N$ thanks to~\eqref{borneLinfini} or Lemma~\ref{IneqLT}.
We recall that $\underline{u}\leq \overline{u}$
by definition and $\underline{u} =\overline{u}$ if and only if $u(x,t)$ converges
locally uniformly in $\R^N$ as $t\to +\infty.$
Therefore, to prove~\eqref{cv-thm-main}, it is enough to prove
$\overline{u}\leq \underline{u}$ in $\R^N.$

A formal direct proof of this inequality is easy: $\overline{u}$ is a subsolution of~\eqref{lc-erprob}, while $\underline{u}$ is a supersolution of~\eqref{lc-erprob}; by Lemma~\ref{barron-jensen}, for any constant $C>0$, $\min\{\overline{u},C\}$ is still a subsolution of~\eqref{lc-erprob} and Lemma~\ref{IneqLT} shows that $\min\{\overline{u},C\}-\underline{u} \to -\infty$ as $|x| \to +\infty$. Moreover $0$ is a strict subsolution of~\eqref{lc-erprob} outside $\mathcal{A}$, therefore by comparison arguments of Ishii \cite{ishii87a} 
$$
\max_{\R^N} \{\min\{\overline{u},C\}-\underline{u}\} \leq \max_{\mathcal{A}} \{\min\{\overline{u},C\}-\underline{u}\}\; ,$$
and letting $C$ tend to $+\infty$ gives $\overline{u}-\underline{u} \leq \max_{\mathcal{A}} \{\overline{u}-\underline{u}\}.$
But the right-hand side is 0
since $u(x,t)$ is decreasing in $t$ on $\mathcal{A}$ using  $H(x,p) \geq 0$ and $l(x)=0$ if $x \in \mathcal{A}$.
This gives the result.

This formal proof, although almost correct, is not correct since we do not have a local uniform convergence of $u$ in a neighborhood of $\mathcal{A}$, in particular because we do not have equicontinuity of the family $\{u(\cdot,t), t\geq 0\}.$ To overcome this difficulty, we use some approximations by inf- and sup-convolutions.

For all $\e >0,$ we introduce
\begin{eqnarray*}
&& u_\e (x,t)=\mathop{\rm inf}_{s\in (0,+\infty)}\{u(x,s)+\frac{|t-s|^2}{\e^2}\},\\
&& u^\e (x,t)=\mathop{\rm sup}_{s\in (0,+\infty)}\{u(x,s)-\frac{|t-s|^2}{\e^2}\}.
\end{eqnarray*}
By~\eqref{borneLinfini}, they are well-defined for all $(x,t)\in\R^N\times [0,+\infty)$
and we have
\begin{eqnarray}\label{ineg-sup-inf241}
&& \mathop{\rm inf}_{\R^N}u_0  \leq u_\e (x,t)\leq u(x,t)\leq u^\e (x,t)\leq v^+(x).
\end{eqnarray}
Notice that the infimum and the supremum are achieved
in $u_\e(x,t)$ and $u^\e(x,t)$ respectively. Moreover Lemma~\ref{IneqLT} still holds for $u_\e$ and $u^\e$.
Taking in the same way the half-relaxed limits for $u_\e$ and $u^\e$,
we obtain (with obvious notations)
\begin{eqnarray*}
  && \mathop{\rm inf}_{\R^N}u_0  \leq \underline{u_\e}
  \leq \underline{u} \leq  \overline{u} \leq \overline{u^\e} \leq v^+
  \quad \text{in $\R^N.$}
\end{eqnarray*}
To prove the convergence result~\eqref{cv-thm-main}, it is therefore sufficient to establish
\begin{eqnarray}\label{but-goal}
\overline{u^\e}\leq \underline{u_\e} \quad \text{in $\R^N,$}
\end{eqnarray}
which is our purpose from now on.
\medskip

The following lemma, the proof of which is standard and left to the reader, collects some useful properties of
$u_\e$ and $u^\e.$
\begin{lemma}\label{prop-convol}\ \\
(i) The functions $u_\e$ and $u^\e$ converge locally
uniformly to $u$ in $\R^N\times [0,+\infty)$ as $\e\to 0.$
\smallskip

\noindent
(ii) The functions $u_\e$ and $u^\e$ are Lipschitz continuous
with respect to $t$ locally uniformly in space, i.e.,
for all $R>0,$ there
exists $C_{\e,R}>0$ such that, for all $x\in B(0,R),$ $t,t'\geq 0,$
\begin{eqnarray}\label{lipt-eps}
&& |u_\e (x,t)-u_\e (x,t')|, |u^\e (x,t)-u^\e (x,t')| \leq C_{\e,R}|t-t'|.
\end{eqnarray}
(iii) For all open bounded subset $\Omega\subset\R^N$,
here exists $t_{\e,\Omega} >0$ with $t_{\e,\Omega}\to 0$ as $\e\to 0,$ such that
  $u_\e$ is solution of~\eqref{lc-cauchy}
  and $u^\e$ is subsolution of~\eqref{lc-cauchy}
  in $\Omega\times (t_{\e,\Omega},+\infty).$
\smallskip

\noindent
(iv) For all $R>0$, there exists $C_{\e,R}, t_{\e, R}>0$ such that,
for all $t>t_{\e,R},$  $u_\e(\cdot,t)$ and $u^\e(\cdot,t)$
are subsolutions of
\begin{eqnarray} \label{sous-sol-enx}
 H(x,Dw(x)) \leq l(x)+c+2C_{\e,R},
 \qquad {\rm in } \ B(0,R).
\end{eqnarray}
Therefore, $u_\e(\cdot,t)$ and $u^\e(\cdot,t)$
are locally Lipschitz continuous in space with a Lipschitz constant independent of $t.$
\end{lemma}

We are now ready to prove that $u_\e(\cdot,t)$ and $u^\e(\cdot,t)$ converge uniformly
on $\mathcal{A}$ as $t\to +\infty.$ We follow the arguments of~\cite{nr99}
(or alternatively, one may use~\cite[Theorem I.14]{cl83}).
We fix $R>0$ such that $\mathcal{A}\subset B(0,R)$ and consider $t_{\e, R}>0$
given by Lemma~\ref{prop-convol}.
Since $w=u_\e$ or $w=u^\e$ is a locally Lipschitz continuous subsolution
of~\eqref{lc-cauchy} in $B(0,R)\times (t_{\e,R},+\infty),$
we have
\begin{eqnarray}\label{lc-eq-lip123}
&& w_t(x,t)\leq w_t(x,t)+H(x,Dw(x,t))\leq l(x),\quad \text{a.e. $(x,t)$,} 
\end{eqnarray}
since $H\geq 0$ by~\eqref{Hnulle}.
Let $x \in \mathcal{A}$, $t>t_{\e,R},$ and $h,r>0.$  We have
\begin{eqnarray*}
&& \frac{1}{|B(x,r)|}\int_{B(x,r)}(w(y,t+h)-w(y,t))dy \\
&=& \frac{1}{|B(x,r)|}\int_{B(x,r)} \int_t^{t+h} w_t(y,s)dsdy\\
&\leq& \frac{1}{|B(x,r)|} \int_{B(x,r)}  \int_t^{t+h}l(y)dsdy
\end{eqnarray*}
Using the continuity of $w, l$ and $l(x)=0,$ and
letting $r\to 0,$ 
we obtain
\begin{eqnarray}\label{decroiss-eps}
w(x,t+h) \leq w(x,t) \qquad \text{for all $x\in \mathcal{A},$ $t>t_{\e,R},$ $h \geq 0$.}
\end{eqnarray}
Therefore $t\mapsto w(x,t)$ is a nonincreasing function on $[t_{\e,R}, \infty),$
  Lipschitz continuous in space on the compact subset $\mathcal{A}$ (uniformly in time)
  and bounded from below according to~\eqref{ineg-sup-inf241}.
  By Dini Theorem, $w(\cdot ,t)$ converges
uniformly on $\mathcal{A}$ as $t\to +\infty$ to
a Lipschitz continuous function. 
Therefore, there exist Lipschitz continuous functions  $\phi_\e,\phi^\e:\mathcal{A}\to\R$ with
$\phi_\e\leq  \phi^\e$ and
\begin{eqnarray*}
  && u_\e(x,t)\to \phi_\e(x), \quad  u^\e(x,t)\to \phi^\e(x), \quad
  \text{uniformly on $\mathcal{A}$ as $t\to +\infty.$}
\end{eqnarray*}

We now use the previous results to prove the convergence of $u$
on $\mathcal{A}.$
By Lemma~\ref{prop-convol}(i), we first obtain that~\eqref{decroiss-eps} holds for $u.$
Therefore $t\mapsto u(x,t)$ is nonincreasing for $x\in \mathcal{A}$,
so $u(\cdot,t)$ converges pointwise as $t\to +\infty$ to some function
$\phi : \mathcal{A}  \to \R.$ 
Notice that we cannot conclude
to the uniform convergence at this step since we do not know
that $\phi$ is continuous.

We claim that $u_\e (x,t),u^\e (x,t)\to \phi(x)$ as $t\to +\infty$, for all $x\in\mathcal{A}$. 
The proof is similar in both cases so we only provide it for $u_\e (x,t)$.
Let $x\in \mathcal{A}$ and $\overline{s}>0$ be such that
\begin{eqnarray}\label{equa173}
&&   u(x,\overline{s})+\frac{|t-\overline{s}|^2}{\e^2} = u_\e (x,t)\leq u(x,t).
\end{eqnarray}
By~\eqref{ineg-sup-inf241},
we have
\begin{eqnarray*}
  && \frac{|t-\overline{s}|^2}{\e^2}\leq v^+(x)- \inf u_0.
\end{eqnarray*}
It follows that $\overline{s}\to +\infty$ as $t\to +\infty$.
Thanks to the pointwise convergence $u(x,s)\to \phi(x)$ as $s\to +\infty,$
sending $t$ to $+\infty$ in~\eqref{equa173},
we obtain
\begin{eqnarray*}
&& \phi(x)+\mathop{\rm lim\,sup}_{t\to +\infty} \frac{|t-\overline{s}|^2}{\e^2}\leq \phi(x),
\end{eqnarray*}
from which we infer $\mathop{\rm lim}_{t\to +\infty} \frac{|t-\overline{s}|^2}{\e^2}=0$.
Therefore, by~\eqref{equa173}, $u_\e (x,t)\to \phi(x)$. The claim is proved,
which implies $\phi_\e=\phi^\e=\phi$ on $\mathcal{A}$.

At this stage, we can apply here the above formal argument to the locally Lipschitz continuous functions $u^\e$ and $u_\e$, noticing that $\overline{u^\e}$ and $\underline{u_\e}$ are also locally Lipschitz continuous functions. We deduce that
$$
\max_{\R^N} \{\min\{\overline{u^\e},C\}-\underline{u_\e}\} \leq \max_{\mathcal{A}} \{\min\{\overline{u^\e},C\}-\underline{u_\e}\}
=\max_{\mathcal{A}} \{\min\{\phi,C\}-\phi\},
$$
and therefore letting $C$ tend to $+\infty$ we have $\overline{u^\e}=\underline{u_\e}$ in $\R^N$.

Recalling that $\underline{u_\e}
  \leq \underline{u} \leq  \overline{u} \leq \overline{u^\e} $ in $\R^N$, we have also $\underline{u} =  \overline{u}$ in $\R^N$, and the conclusion follows, completing the proof of the case when $u_0$ is bounded.
  
We consider now the case when $u_0$ is only bounded from below (but not necessarely from above). We set
$u_0^C = \min\{u_0,C\}$. If $w$ denotes the solution of~\eqref{lc-cauchy} associated to the initial data $0$, then, because of the Barron-Jensen results, the solution associated to 
$u_0^C$ is $\min\{u,w+C\}$.

But, from the first step, we know that (i) $w$ converges locally uniformly to some solution $v_1$ of~\eqref{lc-erprob}, (ii) $\min\{u,w+C\}$ converges to some solution $v^C_2$ of~\eqref{lc-erprob} (depending perhaps on $C$) and (iii) we have \eqref{borneLinfini} for $u$.

Let $\mathcal{K}$ be any compact subset of $\R^N$. If $C$ is large enough in order to have $v_1+C>v^+$ on $\mathcal{K}$ (the size of such $C$ depends only on $\mathcal{K}$),
  then for large $t,$ $\min\{u,w+C\}=u$ on $\mathcal{K}$ by the uniform convergence of $w$ to $v_1$ on $\mathcal{K}$. It follows that $u$ converges locally uniformly to $v_2^C$ on $\mathcal{K},$
  which is independent on $C$.  The proof of Theorem~\ref{lc-tm:longtime} is complete.
\hfill$\Box$
\bigskip

To conclude this section, we point out the following result which is a consequence of the comparison argument we used in the proof.

\begin{corollary}\label{meme-comport}
Assume~\eqref{Hcoerc}-\eqref{Hnulle}-\eqref{Hconvex}-\eqref{Hlip}, $l\geq 0$ and~\eqref{lc-argmin-l}.
Then, for all bounded from below solutions $(0,v_1)$ and $(0,v_2)$ of~\eqref{lc-erprob},
$v_1, v_2\to +\infty$ as $|x|\to +\infty$ and
\begin{eqnarray*}
  \mathop{\rm sup}_{\R^N}\{v_1-v_2\}\leq \mathop{\rm max}_{\mathcal{A}}\{v_1-v_2\}<+\infty.
\end{eqnarray*}
\end{corollary}

\begin{remark} \label{rmk-perio}It is quite surprising that, though a lot of different solutions
  to~\eqref{lc-erprob} may exist (see Section~\ref{sec:exple-sol-ergo}),
  all the bounded from below solutions associated to $c=-{\rm min}\,l$ have the same growth at infinity. This is not true
  when $\mathcal{A}$ is not compact, e.g., in the periodic case. Consider for instance
\begin{eqnarray*}
&& |Dv|=|{\rm sin}(x)|\quad\text{in $\R.$}    
\end{eqnarray*} 
For $c=-\mathop{\rm min}_{\R}|{\rm sin}(x)|=0,$ it is possible
to build infinitely many solutions with very different behaviors
by gluing some branches of cosine functions, see Figure~\ref{dess-period}.
\begin{figure}[ht]
\begin{center}
\includegraphics[width=14cm]{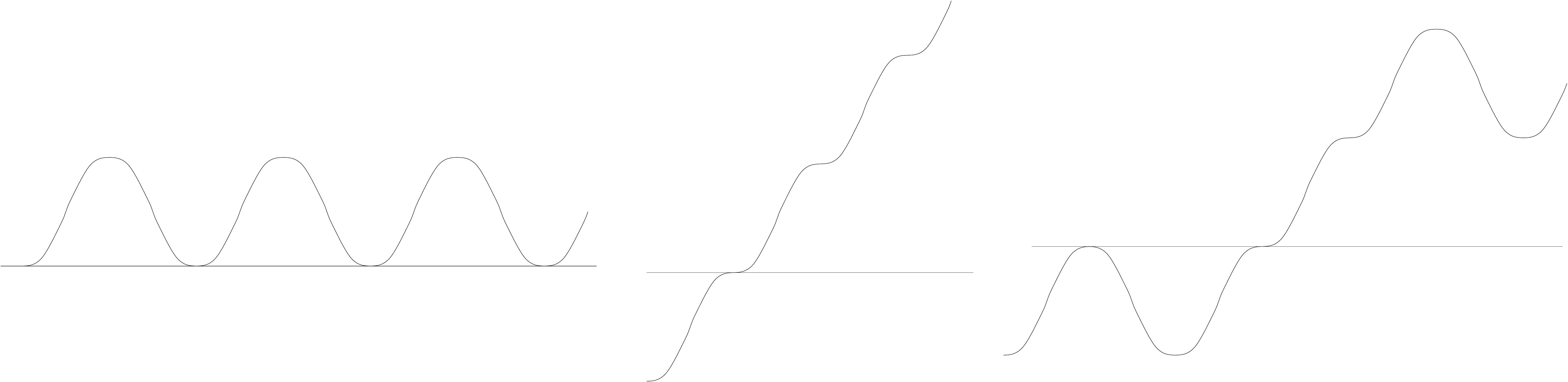}
  \end{center}
\caption{Some solutions of $|Dv|=|{\rm sin}(x)|$ in $\R,$ $\mathcal{A}=\pi\Z$}
\label{dess-period}
\end{figure}
\end{remark}

\subsection{Proof of Theorem~\ref{lc-tm:longtime-infty}}
By~\eqref{u0-v-to0}, there exists a subsolution $(0,\psi)$ of~\eqref{lc-erprob} such that, for every
$\epsilon >0,$ there exists $R_\e >0$ such that, for all  $|x|\geq R_\e,$
\begin{eqnarray}\label{u0vproches}
  &&  {\rm min} \{\psi(x),  v(x)\}-\e \leq  {\rm min} \{\psi(x),  u_0(x)\}
  \leq  {\rm min} \{\psi(x),  v(x)\}+\e.
\end{eqnarray}
Let $M_\e >0$ be such that
\begin{eqnarray*}
&& -M_\e\leq u_0(x), v(x) \qquad \text{for all $|x|\leq R_\e.$}
\end{eqnarray*}
Setting $\psi_\e(x)=  {\rm min} \{\psi(x), -M_\e\},$ we claim that, for all $x\in\R^N,$
\begin{eqnarray}\label{ineg602}
  && {\rm min} \{\psi_\e(x),  v(x)\} -\e
  \leq {\rm min} \{\psi_\e(x),  u_0(x)\}
  \leq  {\rm min} \{\psi_\e(x),  v(x)\} +\e.
\end{eqnarray}
Indeed, this inequality comes from the $M$-property~\eqref{u0vproches} if $|x|\geq R_\e,$ while it is obvious by the choice of $M_\e,$ if $|x|\leq R_\e.$

From Lemma~\ref{barron-jensen},
$(0,\psi_\e)$ is a subsolution of~\eqref{lc-erprob} as a minimum of subsolutions.
Since $c\geq 0$, $(c,\psi_\e)$ is also subsolution of~\eqref{lc-erprob}.
Applying again Lemma~\ref{barron-jensen} to $(c,\psi_\e)$ and $(c,v)$,
we obtain that $(c,{\rm min} \{\psi_\e,  v\})$ is a subsolution of~\eqref{lc-erprob}. 
From~\eqref{ineg602}, we have
${\rm min} \{\psi_\e,  v\} -\e\leq u_0$. It follows from
Proposition~\ref{lc-exis-uniq-cauchy}, that there
exists a unique viscosity solution $u$ of~\eqref{lc-cauchy}
with initial data $u_0$ and it satisfies
\begin{eqnarray}\label{ineg693}
{\rm min} \{\psi_\e,  v\} -\e\leq u(x,t)+ct\leq v^+(x),
\end{eqnarray}
where $(c,v^+)$ is a supersolution of~\eqref{lc-erprob} such that $u_0\leq v^+$.

In the same way, there exists unique viscosity solutions $w_\e$ and $w$ of~\eqref{lc-cauchy}
associated with initial datas $\psi_\e$ and 0 respectively.
Since $\psi_\e\leq -M_\e$, by comparison and Proposition~\ref{lc-exis-uniq-cauchy},
we have
\begin{eqnarray}\label{ineg694}
  && \psi_\e(x)\leq w_\e(x,t)\leq -M_\e +w(x,t)\leq \tilde{v}^+(x)
  \quad \text{for $(x,t)\in\R^N\times [0,+\infty)$,}
\end{eqnarray}
where $(0,\tilde{v}^+)$ is a supersolution of~\eqref{lc-erprob} such that $-M_\e\leq \tilde{v}^+$.

Arguing as at the end of the proof of Theorem~\ref{lc-tm:longtime}, the solutions of~\eqref{lc-cauchy} associated to the initial datas
$$
{\rm min} \{\psi_\e(x),  v(x)\} -\e, \quad
{\rm min} \{\psi_\e(x),  u_0(x)\}
\quad\text{and}\quad
{\rm min} \{\psi_\e(x),  v(x)\} +\e
$$
are respectively
$$
{\rm min} \{w_\e(x,t),  v(x)-ct\} -\e,
\ {\rm min} \{w_\e(x,t),  u(x,t)\}
\ \text{and}\
{\rm min} \{w_\e(x,t),  v(x)-ct\} +\e.
$$

By comparison, we have, in $\R^N\times (0,+\infty)$
$$
{\rm min} \{w_\e(x,t),  v(x)-ct\} -\e \leq {\rm min} \{w_\e(x,t),  u(x,t)\}\; ,$$
and
$$ {\rm min} \{w_\e(x,t),  u(x,t)\} \leq  {\rm min} \{w_\e(x,t),  v(x)-ct\} +\e.
$$
Recalling that $c$ is positive and using~\eqref{ineg694}, if $\mathcal{K}$ is a compact subset of $\R^N$, then
for $t$ large enough and $x \in \mathcal{K}$
$$ {\rm min} \{w_\e(x,t),  v(x)-ct\} = v(x)-ct\; ,$$
leading to the inequality
$$v(x)-ct -\e \leq {\rm min} \{w_\e(x,t),  u(x,t)\} \leq v(x)-ct +\e.
$$
From~\eqref{ineg693} and~\eqref{ineg694},
$t$ can be chosen large enough to have $w_\e(x,t) +ct>u(x,t)+ct $
so we end up with
$$ v(x)-ct -\e \leq u(x,t) \leq v(x)-ct +\e,$$
for $t$ large enough and $x$ in $\mathcal{K}$. Since $\e$ is arbitrary, the conclusion follows.\hfill$\Box$

\begin{remark}\label{rmqII}
Theorem~\ref{lc-tm:longtime-infty} is very close to \cite[Theorem 5.3]{ii09}. In the latter paper, the authors
obtain the convergence assuming that $\mathop{\rm inf}_{\R^N}\{u_0-{\rm min}\{\psi,v \}\}> -\infty$ and
\begin{eqnarray}\label{hypII}
\mathop{\rm lim}_{r\to +\infty}\{|(u_0-v)(x)| : \psi(x)>v(x)+r\}=0.
\end{eqnarray}
We do not know if this assumption is equivalent to ours. But in both assumptions,
the point is that $u_0(x)$ must be close to $v(x)$ when
$v(x)$ is ``far below'' $\psi(x)$, which means
$\psi(x)>v(x)+r$ for large $r$ in~\eqref{hypII} and
${\rm min}\{\psi(x), -r\}> v(x)$ for large $r$ in our case.
This situation occurs for instance if $v,u_0$ are unbounded
from below and close when $v(x)\to -\infty$.
\end{remark}


\section{Optimal control problem and examples}
\label{lc-sec:exple}

Consider the one-dimensional Hamilton-Jacobi Equation
\begin{eqnarray}\label{lc-exp}
\begin{cases}
u_t(x,t) + |Du(x,t)| = 1+ |x| \qquad \text{in } \R \times (0, \infty)\\
u(x,0) = u_0(x),
\end{cases}
\end{eqnarray}
where $l(x)=|x|+1\geq 0,$ $\min_{\R^N} l= 1,$ and $\mathcal{A}={\rm argmin}\,l=\{0\}$
satisfies~\eqref{lc-argmin-l}. We can come back to our framework by looking at $\tilde u (x,t)=u (x,t)-t$ which solves
$$\tilde u_t(x,t) + |D\tilde u(x,t)| =  |x| \qquad \text{in } \R \times (0, \infty)\; ,$$
where $\tilde l(x):= |x|$ satisfies the assumptions of our results.

There exists a unique continuous solution $u$ of~\eqref{lc-exp} for every continuous
$u_0$ satisfying $|u_0(x)|\leq C(1+|x|^2)$ (use Theorem~\ref{lc-compathm} and the fact
that $\pm Ke^{Kt}(1+|x|^2)$ are super- and subsolution for large $K$).

We can represent $u$ as the value function of the following associated deterministic optimal
control problem. Consider the controlled ordinary differential equation
\begin{eqnarray}\label{lc-expcontrol}
\begin{cases}
\dot{X}(s) = \alpha(s),\\
X(0) = x \in \R,
\end{cases}
\end{eqnarray}
where the control $\alpha(\cdot)\in L^\infty([0,+\infty); [-1,1])$
(i.e.,  $|\alpha(t)| \leq 1$ a.e. $t\geq 0$).
For any given control $\alpha$,~\eqref{lc-expcontrol} has a unique solution 
$X(t) =X_{x,\alpha (\cdot)}= x + \int_0^t \alpha(s)ds$.
We define the cost functional
$$J(x,t,\alpha) = \int_0^t(|X(s)| + 1)ds + u_0(X(t)),$$
and the value function 
$$V(x,t) = \inf_{\alpha\in L^\infty([0,+\infty); [-1,1])} J(x,t,\alpha).$$
It is classical to check that $V(x,t)=u(x,t)$ is the unique viscosity solution of~\eqref{lc-exp},
see~\cite{barles94, bcd97}.

\subsection{Solutions to the ergodic problem}
\label{sec:exple-sol-ergo}

There are infinitely many essentially different solutions with
different constants to the associated
ergodic problem
\begin{eqnarray}\label{exple-ergo}
|Dv(x)|=1+|x|+c \qquad \text{in } \R.
\end{eqnarray}
Define $S(x)=\int_0^x |y|dy.$
The following pairs $(c,v)$ are solutions.
\begin{itemize}

\item $(-1, \frac{1}{2}x^2)$ and $(-1, -\frac{1}{2}x^2).$ They are bounded from below 
(respectively from above) with $c=-\min\,l$;

\item $(-1,S(x))$ and  $(-1,-S(x)).$ They are neither bounded from below nor from above and
$c=-\min\,l$;

\item $(\lambda -1, \lambda x+S(x))$ and $(\lambda -1, -\lambda x-S(x))$
for every $\lambda >0.$ They are 
neither bounded from below nor from above and
$c>-\min\,l$;

\item $(\lambda -1, -\frac{1}{2}x^2-\lambda|x|)$ for every $\lambda >0.$
  These solutions are nonsmooth (notice that $-v$ is not anymore a viscosity
  solution), they are not bounded from below. Actually, they are the
  solutions obtained by the constructive proof of Theorem~\ref{lc-thm:ergodic}(i).
  Indeed, the unique solution $V_R$ of the Dirichlet problem~\eqref{pb-dirichl}
  is $V_R(x)= \frac{R^2-x^2}{2}+\lambda(R-|x|)$ for $x\in [-R,R],$
  leading to $v(x)={\rm lim}_{R\to \infty}\{V_R(x)-V_R(0)\}= -\frac{1}{2}x^2-\lambda|x|.$

\item $(c,v)$ where $(c,v_1)$ and $(c,v_2)$ are solutions, 
  $v=\min \{v_1+C_1, v_2+C_2\}$ and $C_1,C_2\in\R.$
  This is a consequence of Lemma~\ref{barron-jensen}.
  
\end{itemize}

\subsection{Equation~\eqref{lc-exp} with $u_0(x)=S(x)$}
\label{lc-exple:S}

For any solution $(c,v)$ to~\eqref{exple-ergo}, it is obvious that
$u(x,t)= -ct+v(x)$ is the unique solution to~\eqref{lc-exp} with $u_0(x)=v(x)$
and the convergence holds, i.e., $u(x,t)+ct\to v(x)$ as $t\to +\infty.$
In particular, if $u_0(x)=S(x),$ the solution of~\eqref{lc-exp} is $u(x,t)= t+S(x).$

Let us find in another way the solution by computing the value function
of the control problem stated above. Let $t>0.$ We compute $V(x,t)$
for any $x\in\R$ by determining the optimal controls and trajectories.

\medskip
\noindent{\it 1st case:} $x\geq 0.$\\
There are infinitely many optimal strategies: they consist
in going as quickly as possible to 0 ($={\rm argmin}\,l$), to wait at 0
for a while and to go as quickly as possible towards $-\infty.$
For any $0\leq \tau \leq t-x,$ it corresponds to the optimal controls
and trajectories
\begin{eqnarray}\label{lc-optS}
&& \scriptsize\alpha(s)= \left\{\begin{array}{cl} 
-1, & 0\leq s\leq x,\\
0, & x\leq s\leq x+\tau,\\
-1, & x+\tau\leq s\leq t,
\end{array}\right.
\qquad
\scriptsize X(s)= \left\{\begin{array}{cl} 
x-s, & 0\leq s\leq x,\\
0, & x\leq s\leq x+\tau,\\
-(s-x-\tau), & x+\tau\leq s\leq t.
\end{array}\right.
\end{eqnarray}
They lead to $V(x,t)=J(x,t,\alpha)=t+S(x).$
Among these optimal strategies, there are two of particular interest:

\begin{itemize}
\item The first one is to go as quickly as possible to 0
and to remain there ($\tau=t-x$). This strategy is typical of
what happens in the periodic case: the optimal trajectories
are attracted by $\mathcal{A}={\rm argmin}\,l.$

\item The second one is to go as quickly as possible towards $-\infty$
during all the available time $t$ ($\tau=0$). This situation is very different
to the periodic case. Due to the unbounded (from below) final
cost $u_0,$ some optimal trajectories are not anymore atttracted
by ${\rm argmin}\,l$ and are unbounded.
\end{itemize}

\medskip
\noindent{\it 2nd case:} $x< 0.$\\
In this case there is not anymore bounded optimal trajectories.
The only optimal strategy is to go as quickly as possible
towards $-\infty.$
The optimal control are $\alpha(s)=-1,$ $X(s)=x-s$
for $0\leq s\leq t$ leading to $V(x,t)=J(x,t,-1)=t+S(x).$

The analysis of this case in terms of control will
help us for the following examples.

\subsection{Equation~\eqref{lc-exp} with $u_0(x)=\frac{1}{2}x^2+b(x)$ with $b$ bounded from below}
\label{lc-exple:carreb}

To illustrate  Theorem~\ref{lc-tm:longtime}, we choose an initial condition
which is a bounded perturbation
of a bounded from below solution of the ergodic problem.
To simplify the computations, we choose a periodic perturbation $b$.

For any $x$, an optimal strategy can be chosen among
those described in Example~\ref{lc-exple:S}. More precisely:
go as quickly as possible to 0,
wait nearly until time $t$ and move a little to reach the minimum
of the periodic perturbation.
For $t$ large enough (at least $t>x$), we compute the cost with
$\alpha, X$ given by~\eqref{lc-optS},
$$
J(x,t,\alpha)= t+\frac{1}{2}x^2+b(-t+x+\tau).
$$
For every $t$ large enough, there exists $0\leq \tau=\tau_t <t-x$ such that
$b(-t+x+\tau_t)=\min \, b.$ It leads to
$$
V(x,t)= J(x,t,\alpha)= t+\frac{1}{2}x^2 + \min \, b.
$$
Therefore, we have the convergence as announced in Theorem~\ref{lc-tm:longtime}.

\subsection{Equation~\eqref{lc-exp} with $u_0(x)=S(x)+b(x)$ with $b$ bounded Lipschitz continuous}

We compute the value function as above. Due to the
unboundedness from below of $u_0$ we need to distinguish
the cases $x\geq 0$ and $x<0$ as in Example~\ref{lc-exple:S}.

\medskip
\noindent{\it 1st case:} $x\geq 0.$\\
We use the same strategy as in Example~\ref{lc-exple:carreb}
leading to $V(x,t)= J(x,t,\alpha)= t+\frac{1}{2}x^2 + \min \, b.$

\medskip
\noindent{\it 2nd case:} $x< 0.$\\
In this case, the optimal strategy suggested by Examples~\ref{lc-exple:S} and~\ref{lc-exple:carreb}
is to start by waiting a small time $\tau$ before going as quickly as possible towards $-\infty.$
The waiting time correspond to an attempt to reach a minimum of $b$ at the
left end of the trajectory.
It corresponds to the control and trajectory
\begin{eqnarray*}
&&\scriptsize\alpha(s)= \left\{\begin{array}{cl} 
0, & 0\leq s\leq \tau,\\
-1, & \tau\leq s\leq t,
\end{array}\right.
\qquad
\scriptsize X(s)= \left\{\begin{array}{cl} 
x, & 0\leq s\leq \tau,\\
x-(s-\tau), & \tau\leq s\leq t,
\end{array}\right.
\end{eqnarray*}
leading to
$$
J(x,t,\alpha)=t+S(x)+\tau |x|+b(x-t+\tau).
$$
Due to the boundedness of $b,$ in order to be optimal,
we see that necessarily $\tau=O(1/|x|)$
to keep bounded the positive term $\tau |x|$ in $J(x,t,\alpha)$. 
So, for large $|x|,$ $x<0,$ we have
$b(x-t+\tau)\approx b(x-t).$
When $b$ is not constant,
$b(x-t)$ has no limit as $t\to +\infty,$
the convergence for $V(x,t)$
cannot hold.

In this case, $u_0$ is a bounded perturbation of
a solution $(c,v)=(-1,S(x))$ of the ergodic problem
with $c = -\min \,l$ but $v$ is not bounded from
below and the convergence of the value function may not hold.
It follows that the assumptions of Theorem~\ref{lc-tm:longtime}
cannot be weakened
easily. In particular, 
the boundedness from below of the 
solution of the ergodic problem seems to be crucial.

\subsection{Equation~\eqref{lc-exp} with $u_0(x)=S(x)+x+{\rm sin}(x)$}

The
solution of~\eqref{lc-exp} is $u(x,t)=S(x)+x+{\rm sin}(x-t).$
Clearly, we do not have the convergence. In this case,
$u_0$ is a bounded perturbation of the solution $(0, S(x)+x)$
of the ergodic problem with
$c> -\min \,l$ and $S(x)+x-u_0(x)\not\to 0$ as $x\to -\infty$
(where $S(x)+x\to -\infty$).
This example shows that the convergence in Theorem~\ref{lc-tm:longtime-infty}
may fail when~\eqref{u0-v-to0} does not hold.

\appendix
\section{Comparison principle for the solutions of~\eqref{lc-cauchy}}

The comparison result for the unbounded solutions of~\eqref{lc-cauchy}
is a consequence of a general comparison result for first-order
Hamilton-Jacobi equations which holds without growth conditions
at infinity.

\begin{theorem} \label{lc-compathm} 
\cite{ishii84, ley01}
Assume that $H$ satisfies~\eqref{Hlip} and that
$u \in USC(\R^N \times [0, T])$ and $v \in LSC(\R^N \times [0, T])$
are respectively a subsolution of~\eqref{lc-cauchy} with initial data $u_0\in C(\R^N)$
and a supersolution of~\eqref{lc-cauchy} with initial data $v_0\in C(\R^N).$
Then, for every $x_0\in\R^N$ and $r>0,$
\begin{eqnarray*}
&& u(x,t) - v(x,t) \leq \sup_{\overline{B}(x_0,r)}\{u_0(y) - v_0(y) \}
\quad \text{for every $(x,t) \in \bar{\mathcal{D}}(x_0,r)$,}
\end{eqnarray*}
where
$$
\bar{\mathcal{D}}(x_0,r) = \lbrace (x,t) \in B(x_0,r) \times (0,T): e^{C_H T}(1 + |x-x_0|) - 1 \leq r \rbrace.
$$
\end{theorem} 
When $\sup_{\R^N} \{u_0-v_0\} < +\infty$, a straightforward consequence is
\begin{eqnarray*}
&& u(x,t) - v(x,t) \leq \sup_{\R^N}\{u_0 - v_0 \}
\quad \text{for every $(x,t) \in \R^N\times [0,+\infty)$.}
\end{eqnarray*}

\section{Barron-Jensen solutions of convex HJ equations}

\begin{theorem} \label{thm-bj}
  Assume that $H$ satisfies~\eqref{Hconvex} and~\eqref{Hlip}. Then $u\in W_{\rm loc}^{1,\infty}(\R^N\times (0,+\infty))$
  is a viscosity
  solution (respectively subsolution) of~\eqref{lc-cauchy} if and ony if it is a Barron-Jensen solution (respectively subsolution)
  of~\eqref{lc-cauchy}, i.e.,
  for every $(x,t)\in \R^N\times (0,+\infty)$ and $\varphi\in C^1(\R^N\times (0,+\infty))$ such that
  $u-\varphi$ has a local minimum at $(x,t),$ one has
  \begin{eqnarray*}
&& \varphi_t(x,t)+ H(x,D\varphi(x,t))=l(x) \ \ \text{(respectively $\leq l(x)$)}.
  \end{eqnarray*}
\end{theorem} 

This result is due to Barron and Jensen~\cite{bj90}
and we refer to Barles~\cite[p. 89]{abil13}.
Lemmas~\ref{prop-convol}(iii) and~\ref{barron-jensen}
are consequences of this theorem.

As far as Lemma~\ref{prop-convol}(iii) is concerned,
the fact that the inf-convolution (respectively the sup-convolution)
preserves the supersolution (respectively the subsolution) property
is classical (\cite{barles94, bcd97}). What is more suprising is
the preservation of the subsolution property of the inf-convolution
which comes from the convexity of $H$ and the Theorem of
Barron-Jensen~\ref{thm-bj}.
For a proof, notice first that $U$, being a solution of~\eqref{lc-cauchy},
is a Barron-Jensen solution of~\eqref{lc-cauchy}. We then
apply~\cite[Lemma 3.2]{ley01} using that $H,l$ are independent of $t.$

For Lemma~\ref{barron-jensen}, we refer the reader to~\cite[Theorem 9.2, p.90]{abil13}.


\begin{thebibliography}{10}

\bibitem{abil13}
Yves Achdou, Guy Barles, Hitoshi Ishii, and Grigory~L. Litvinov.
\newblock {\em Hamilton-{J}acobi equations: approximations, numerical analysis
  and applications}, volume 2074 of {\em Lecture Notes in Mathematics}.
\newblock Springer, Heidelberg; Fondazione C.I.M.E., Florence, 2013.
\newblock Lecture Notes from the CIME Summer School held in Cetraro, August
  29--September 3, 2011, Edited by Paola Loreti and Nicoletta Anna Tchou,
  Fondazione CIME/CIME Foundation Subseries.

\bibitem{bcd97}
M.~Bardi and I.~Capuzzo~Dolcetta.
\newblock {\em Optimal control and viscosity solutions of
  {H}amilton-{J}acobi-{B}ellman equations}.
\newblock Birkh\"auser Boston Inc., Boston, MA, 1997.

\bibitem{barles94}
G.~Barles.
\newblock {\em Solutions de viscosit\'e des \'equations de
  {H}amilton-{J}acobi}.
\newblock Springer-Verlag, Paris, 1994.

\bibitem{br06}
G.~Barles and J.-M. Roquejoffre.
\newblock Ergodic type problems and large time behaviour of unbounded solutions
  of {H}amilton-{J}acobi equations.
\newblock {\em Comm. Partial Differential Equations}, 31(7-9):1209--1225, 2006.

\bibitem{bs00}
G.~Barles and P.~E. Souganidis.
\newblock On the large time behavior of solutions of {H}amilton-{J}acobi
  equations.
\newblock {\em SIAM J. Math. Anal.}, 31(4):925--939 (electronic), 2000.

\bibitem{bim13}
Guy Barles, Hitoshi Ishii, and Hiroyoshi Mitake.
\newblock A new {PDE} approach to the large time asymptotics of solutions of
  {H}amilton-{J}acobi equations.
\newblock {\em Bull. Math. Sci.}, 3(3):363--388, 2013.

\bibitem{bm16}
Guy Barles and Joao Meireles.
\newblock On unbounded solutions of ergodic problems in {$\Bbb{R}^m$} for
  viscous {H}amilton-{J}acobi equations.
\newblock {\em Comm. Partial Differential Equations}, 41(12):1985--2003, 2016.

\bibitem{bj90}
E.~N. Barron and R.~Jensen.
\newblock Semicontinuous viscosity solutions of {H}amilton-{J}acobi equations
  with convex hamiltonians.
\newblock {\em Comm. Partial Differential Equations}, 15(12):1713--1740, 1990.

\bibitem{cil92}
M.~G. Crandall, H.~Ishii, and P.-L. Lions.
\newblock User's guide to viscosity solutions of second order partial
  differential equations.
\newblock {\em Bull. Amer. Math. Soc. (N.S.)}, 27(1):1--67, 1992.

\bibitem{cl83}
M.~G. Crandall and P.-L. Lions.
\newblock Viscosity solutions of {H}amilton-{J}acobi equations.
\newblock {\em Trans. Amer. Math. Soc.}, 277(1):1--42, 1983.

\bibitem{dalio02}
F.~Da~Lio.
\newblock Comparison results for quasilinear equations in annular domains and
  applications.
\newblock {\em Comm. Partial Differential Equations}, 27(1-2):283--323, 2002.

\bibitem{ds06}
A.~Davini and A.~Siconolfi.
\newblock A generalized dynamical approach to the large time behavior of
  solutions of {H}amilton-{J}acobi equations.
\newblock {\em SIAM J. Math. Anal.}, 38(2):478--502 (electronic), 2006.

\bibitem{fathi98}
A.~Fathi.
\newblock Sur la convergence du semi-groupe de {L}ax-{O}leinik.
\newblock {\em C. R. Acad. Sci. Paris S\'er. I Math.}, 327(3):267--270, 1998.

\bibitem{fs04}
A.~Fathi and A.~Siconolfi.
\newblock Existence of {$C^1$} critical subsolutions of the {H}amilton-{J}acobi
  equation.
\newblock {\em Invent. Math.}, 155(2):363--388, 2004.

\bibitem{fs05}
A.~Fathi and A.~Siconolfi.
\newblock P{DE} aspects of {A}ubry-{M}ather theory for quasiconvex
  {H}amiltonians.
\newblock {\em Calc. Var. Partial Differential Equations}, 22(2):185--228,
  2005.

\bibitem{fm07}
Albert Fathi and Ezequiel Maderna.
\newblock Weak {KAM} theorem on non compact manifolds.
\newblock {\em NoDEA Nonlinear Differential Equations Appl.}, 14(1-2):1--27,
  2007.

\bibitem{ii08}
Naoyuki Ichihara and Hitoshi Ishii.
\newblock The large-time behavior of solutions of {H}amilton-{J}acobi equations
  on the real line.
\newblock {\em Methods Appl. Anal.}, 15(2):223--242, 2008.

\bibitem{ii09}
Naoyuki Ichihara and Hitoshi Ishii.
\newblock Long-time behavior of solutions of {H}amilton-{J}acobi equations with
  convex and coercive {H}amiltonians.
\newblock {\em Arch. Ration. Mech. Anal.}, 194(2):383--419, 2009.

\bibitem{ishii84}
H.~Ishii.
\newblock Uniqueness of unbounded viscosity solution of {H}amilton-{J}acobi
  equations.
\newblock {\em Indiana Univ. Math. J.}, 33(5):721--748, 1984.

\bibitem{ishii87a}
H.~Ishii.
\newblock A simple, direct proof of uniqueness for solutions of the
  {H}amilton-{J}acobi equations of eikonal type.
\newblock {\em Proc. Amer. Math. Soc.}, 100(2):247--251, 1987.

\bibitem{ishii08}
H.~Ishii.
\newblock Asymptotic solutions for large time of {H}amilton-{J}acobi equations
  in {E}uclidean {$n$} space.
\newblock {\em Ann. Inst. H. Poincar\'e Anal. Non Lin\'eaire}, 25(2):231--266,
  2008.

\bibitem{ishii09}
H.~Ishii.
\newblock Asymptotic solutions of {H}amilton-{J}acobi equations for large time
  and related topics.
\newblock In {\em I{CIAM} 07---6th {I}nternational {C}ongress on {I}ndustrial
  and {A}pplied {M}athematics}, pages 193--217. Eur. Math. Soc., Z\"urich,
  2009.

\bibitem{ley01}
O.~Ley.
\newblock Lower-bound gradient estimates for first-order {H}amilton-{J}acobi
  equations and applications to the regularity of propagating fronts.
\newblock {\em Adv. Differential Equations}, 6(5):547--576, 2001.

\bibitem{lpv86}
P.-L. Lions, B.~Papanicolaou, and S.~R.~S. Varadhan.
\newblock Homogenization of {H}amilton-{J}acobi equations.
\newblock {\em Unpublished}, 1986.

\bibitem{nr99}
G.~Namah and J.-M. Roquejoffre.
\newblock Remarks on the long time behaviour of the solutions of
  {H}amilton-{J}acobi equations.
\newblock {\em Comm. Partial Differential Equations}, 24(5-6):883--893, 1999.

\end{thebibliography}


\end{document}